\documentclass[a4paper, 14pt]{article}

\usepackage[utf8]{inputenc}
\usepackage[T2A]{fontenc}
\usepackage[russian]{babel}
\usepackage{stmaryrd}
\usepackage{graphicx}
\usepackage{amssymb,amsmath,mathrsfs,color}
\usepackage{cancel}
\usepackage{hyperref}

\usepackage{algpseudocode}
\usepackage{todonotes} % заметки снизу

\usepackage{amsthm}

\newtheorem{theorem}{Теорема}%[section]
\newtheorem{lemma}{Лемма}%[section]
\usepackage{extsizes} % Возможность сделать 14-й шрифт
\usepackage{geometry} % Простой способ задавать поля
\newcommand\cF{{\mathcal F}}

\usepackage{graphicx, amssymb, color}
\usepackage{hyperref}

\newcommand\Prop{{\mathsf P}}
\newcommand\E{{\mathsf E}}

\newcommand\R{{\mathbb R}}
\newcommand\N{{\mathbb N}}

\newcommand\e{{\varepsilon}}

\newcommand\beq{\begin{equation}}
\newcommand\eeq{\end{equation}}
\newcommand\bea{\begin{eqnarray}}
\newcommand\eea{\end{eqnarray}}
\newcommand\bean{\begin{eqnarray*}}
\newcommand\eean{\end{eqnarray*}}

\usepackage{marginnote}

\begin{document}
\title{
%Positive recurrence of an SDE with additive Wiener process and variable switching intensities
%Нахождение достаточные условия для
О свойстве транзиентности и некоторых оценках для процесса диффузии с переключением
}
\author{\text{Кирилл Мосиевич}\footnote{
Московский государственный университет имени М.В.Ломоносова, Москва, Российская Федерация, email: {kirillmosievich@gmail.com}}\;\footnote{Работа выполнена при финансовой поддержке Фонда развития теоретической физики и математики «БАЗИС»,  21-8-2-32-1.
}
}
\maketitle  
% \begin{titlepage}
% 	\begin{center}
% 		%\large
% 		ФЕДЕРАЛЬНОЕ ГОСУДАРСТВЕННОЕ БЮДЖЕТНОЕ ОБРАЗОВАТЕЛЬНОЕ\\
% 		УЧРЕЖДЕНИЕ ВЫСШЕГО ОБРАЗОВАНИЯ
		
% 		<<МОСКОВСКИЙ ГОСУДАРСТВЕННЫЙ УНИВЕРСИТЕТ\\
% 		им. М. В. ЛОМОНОСОВА>>
		
% 		\vspace{0.25cm}
		
% 		МЕХАНИКО-МАТЕМАТИЧЕСКИЙ ФАКУЛЬТЕТ
		
% 		\vspace{0.25cm}
		
% 		КАФЕДРА ТЕОРИИ ВЕРОЯТНОСТЕЙ
		
% 		\vfill
% 		\textsc{ВЫПУСКНАЯ КВАЛИФИКАЦИОННАЯ РАБОТА\\
% 				(ДИПЛОМНАЯ РАБОТА)
% 				}
			
% 			специалиста\\[5mm]
		
% 		{\Large 
%             \textbf{
% 		О свойстве транзиентности и некоторых оценках для процесса диффузии с переключением
% 		}
%   }
% 	\end{center}
% 	\vfill
	
% 	%\newlength{\ML}
% 	%\settowidth{\ML}{«\underline{\hspace{0.7cm}}» \underline{\hspace{2cm}}}
% 	\hfill\begin{minipage}{0.45\textwidth}
% 		Выполнил студент\\
% 		608 группы\\
% 		Мосиевич Кирилл Викторович\\
% 		\\
% 		\underline{\hspace{5cm}}\\
% 		подпись студента\\
% 		\\
% 		Научный руководитель:\\
% 		д.ф.-м.н., профессор\\ Веретенников Александр Юрьевич\\
% 		\\
% 		\underline{\hspace{5cm}}\\
% 		подпись научного руководителя\\
% 	\end{minipage}
% 	\bigskip
	
% 	\vfill
	
% 	\begin{center}
% 		Москва
% 	\end{center}
%         \begin{center}
%             2024
%         \end{center}
% \end{titlepage}

% \tableofcontents
% Содержание
% \newpage
\begin{abstract}
Данная работа показывает ограничения на условия для экспоненциальной эргодичности с данной системой переключения.
\par
В работе установлены достаточные условия транзиентности процесса для модели марковской диффузии с переключениями и двумя режимами, транзиентным и эргодическим, при интенсивностях строго отделимых от нуля. Также получены экспоненциальные оценки вероятности ухода процесса на бесконечность.  
% Установлены достаточные условия транзиентности процесса для модели марковской диффузии с переключениями и двумя режимами, транзиентным и эргодическим, при интенсивностях строго отделимых от нуля. Данная работа показывает ограничения на условия для экспоненциальной эргодичности с данной системой переключения.\\
~
{\em Ключевые слова: {диффузионные процессы, переключение, постоянная интенсивность переключения, транзиентность}}
~
{\em MSC коды: 60H10, 60J60}
\end{abstract}

\newpage
\addcontentsline{toc}{section}{Введение}
\section*{Введение}
% \section{Введение}
\par
Диффузионные процессы с переключениями встречаются во многих областях исследований, таких как моделирование биологических систем \cite{link12, link13}, моделирование систем хранения \cite{link14}. Такие модели содержат две компоненты $(X_t, Z_t)$. Первая компонента $X_t$ используется для описания исследуемой динамической системы, второй компонент $Z_t$ используется для описания случайного изменения среды, в которой находится динамическая система. Поскольку эти модели учитывают изменения окружающей среды, они могут более точно моделировать практические задачи.
\par
Нахождению достаточных условий для эргодичности процесса с переключениями были посвящены работы: \cite{link5} \cite{link6}, \cite{link7}. В работе \cite{link5} рассматривается экспоненциальный случай, работа \cite{link6} посвящена полиномиальному случаю. Помимо изучения достаточных условий эргодичности, важно исследовать транзиентный случай, по которому можно судить об оптимальности полученных результатов в эргодическом случае. В этой работе берутся обратные, но более сильные требования (остается некоторый "зазор"), к условиям, введенным в работе \cite{link5}. Что также показывает сложность в дальнейшем "улучшении"\, условий в работе \cite{link5}.

\addcontentsline{toc}{subsection}{Постановка задачи}
\subsection*{Постановка задачи}
Зададим вероятностное пространство $(\Omega, \cF,\cF_t, \Prop)$ с фильтрацией \\
$\cF_t = \sigma(\cF^W_t\sqcup\cF^Z_t)$, где $\cF^W_t$ - естественная фильтрация одномерного виноровского процесса $W = (W_t)_{t\geq0}$, $\cF^Z_t$ - естественная фильтрация, $Z=(Z_t)_{t\ge0}$, марковского процесс с бинарным множеством значений $S = \{0, 1\}$, с положительными постоянными и конечными интенсивностями переключения: $\lambda_{0} =: \lambda_{+}, \; \lambda_{1} =: \lambda_{-}$. При чем процессы $Z_t$ имеет независимые с $W_t$ распределения, и моменты переключения $Z_t$ распределены экспоненциально.
\par
Рассмотрим стохастический процесс с переключением $(X_t, Z_t)$ с непрерывной компонентой $X$ и дискретной компонентой $Z$. Причем стохастическое дифференциальное уравнение на компоненту $X$ имеет вид:
\begin{equation}\label{SDE}
dX_t = b(X_t, Z_t) dt + dW_t,\; t\geq0 \; , X_0 = x, \;, Z_0 = z.\\
\end{equation}
Траектории процесса $Z$ непрерывные справа имеют предел слева, так как фазовое пространство процесса дискретно; распределения прыжков процесса $Z$ независимы при фиксированной компоненты $x$.
% В режиме $Z_t = 0$ процесс $X_t$ "транзиентен"\ ,  в режиме $Z_t = 1$ процесс "рекуррентен". 
\par
Обозначим\\
$$
    b(x, 0) = b_-(x), \;\; b(x, 1) = b_+(x),
$$
Положим также, что  снос процесса $X$ ограничен, то есть, 
\begin{eqnarray*}
 b_+ &:=& ||b_+(x)|| = \sup _{x} |b_+(x)| < \infty,\\
 b_- &:=& ||b_-(x)|| = \sup _{x} |b_-(x)| < \infty.
\end{eqnarray*}
Тогда процесс  $(X_t, Z_t)$ корректно определен и задается системой,
\begin{align*}
    dX_t = b(X_t, Z_t) dt + dW_t,\; t\geq0, \; X_0 = x \in R,\\
    dZ_t = 1(Z_t = 0)d\pi^0_t + 1(Z_t = 1)d\pi^1_t, \; Z_0 \in \{0, 1\},
\end{align*}
где $\pi^i_t,\, i\in \{0,1\}$ - два пуассоновских процесса с интенсивностями $\lambda_+, \lambda_-$.\\
Более точно,
\begin{align}
\pi^i_t = \overline{\pi}^i_{\phi_i(t)},
\end{align}
где $\overline{\pi}^i_{\phi(t)},\, i\in\{0,1\}$ - являются, в свою очередь, двумя пуассоновскими процессами с постоянной интенсивностью, скачки которых не зависят от винеровского процесса $W_t$ и друг от друга.
\begin{align*}
t \to \phi_i(t) := \int\limits_0^t \lambda_i\, ds,\; i=0,1
\end{align*}
замена времени, применяемая к каждому из них соответственно.
\par
В условиях ограниченности сноса и его согласованности по паре переменных существование и единственность сильного решения уравнения $(\ref{SDE})$ следует из работы $\cite{link1}$, или $\cite{link2}$, или $\cite{link3}$.
\par
Рассмотрим моменты переключения процесса $Z$:
\begin{align}{\label{T_2n}}
T_0 := \inf\left(t\geq0: Z_t = 0\right),
0 \leq T_0 < T_1 < T_2 < \dots \;, 
\end{align}
где $T_n := \inf\left(t > T_{n-1}, Z_t \not= Z_{T_{n-1}} \right)$.\\
Они являются моментами остановки относительно фильтрации $\cF_t = \cF_t^{W, \pi^0, \pi^1},\, t\ge0$. Также заметим, что решение уравнения ($\ref{SDE}$) имеет предел слева, то есть положение системы в момент переключения $(X_{T_{2n}},Z_{T_{2n}})$ однозначно задается левым предельным значением $(X_{T_{2n-}}, Z_{T_{2n}-})$. Поэтому процесс $(X, Z)$ марковский. Его генератор имеет вид: 
\begin{align*}
Lh(x,z) = \frac{1}{2} \frac{\partial^2}{\partial x^2} + b(x,z)\frac{\partial}{\partial x} + \lambda_z(h(x,z) - h(x, \overline{z})),
\end{align*}
где $\overline{z} = 1\{z = 0\},\, z\in\{0, 1\}$.\\
Покажем, что процесс Феллеровский, откуда согласно работе \cite{link8} он будет обладать сильно марковским свойством.
\par
Для любого $t > 0$ положим,
\begin{align*}
u(s, x, z) := \E_{s,x,z}f(X_t, Z_t),
\end{align*}
где функция $f$ - произвольная непрерывная и ограниченная функция. Заметим, что $z$ - дискретно, то нам нужно проверить только непрерывность по паре переменных $(s,x)$.
\par
Тогда вектор-функция $u(s,x) := (u(s,x,0), u(s,x,1))$ удовлетворяет системе дифференциальных уравнений,
\begin{align}
    \begin{cases}{\label{PDE}}
        u_s(s,x,0) + \frac{1}{2}u_{xx}(s,x,0) + b(x,0)u_x(s,x,0) + \lambda_0 (u(s,x,1) - u(s,x,0)) = 0,\;\\
        u(t,x,0) = f(x,0)\\
        u_s(s,x,1) + \frac{1}{2}u_{xx}(s,x,1) + b(x,1)u_x(s,x,1) + \lambda_1 (u(s,x,0) - u(s,x,1)) = 0,\;\\ u(t,x,1) = f(x,1)
    \end{cases}
\end{align}
Из работы [\cite{link10}, теорема 5.5] следует, что решение системы (\ref{PDE}) существует единственное и лежит локально в Соболевском классе $\mathcal{W}^{1,2}_p$ с любым $p>1$, так как $u(s,x,z)$ ограниченная. Откуда, согласно лемме о вложении [\cite{link11}, лемма 3.3], решение задается непрерывной функцией $u(s, x, z)$, что и дает нам феллеровость процесса $(X, Z)$.

\newpage
\addcontentsline{toc}{section}{Основные результаты}
\section*{Основные результаты}

\begin{theorem}\label{t1}
Пусть снос $b_t:= b(x, z)$ ограничен и согласован с фильтрацией $\cF_t$, и измерим по паре $(x, z)$, и существуют $ r_-, \; r_+$ строго больше нуля такие, что для любого $x \in \R$,
$$
b_-(x) \geq -r_- , \; b_+(x) \geq r_+
$$
и
$$
\frac{r_+}{\lambda_+ } 
> \frac{r_-}{\lambda_- }.
$$
\par
Тогда процесс $X_t \to \infty$ \footnote{Так как снос процесса $X_t$ ограничен, можно также показать, что уход процесса на бесконечность не может быть более чем линейным.} почти наверное, при $t \to \infty$. 
Причем найдется постоянная $c_0 > 0$, такая, что
\begin{align}{\label{t1_1}}
\liminf_{ t \to \infty}\frac{X_t}{t} \ge c_0, \quad\text{п.н.}
\end{align}
% {\color{blue} Добавить замечание, что уход не может быть более чем линейный, так как снос ограничен.}
\par
Более того, верна экспоненциальная оценка.
% $\exists C, K, c_0, c, t_0 > 0:\, \forall \epsilon > 0\, \exists 0 < \kappa_3, \lambda < 1,\,\\
% \forall t > t_0(C, K, \kappa_3, \lambda, c) $:\\
Существует положительная константа $c_0$, такая, что для любого $\epsilon > 0$ найдется константа $0 < \kappa < 1$, такая, что для достаточно больших $t \ge t_0(\kappa, \epsilon)$:
\begin{align}{\label{t1_2}}
\Prop_x\left(\frac{X_{t} - x}{t} - c_0 < -\epsilon \right) < \kappa^{t} < 1.
\end{align}
\end{theorem}
Напомним определение $T_{2n}$, 
\begin{align*}
0 \leq T_0 < T_1 < T_2 < \dots
\end{align*}
Где $T_0 := \inf\left(t\geq0: Z_t = 0\right)$, $T_n := \inf\left(t > T_{n-1}, Z_t \not= Z_{T_{n-1}} \right)$.

\begin{lemma}\label{l1}
В условиях теоремы при $n \to \infty$,
\begin{align}{\label{T_n_to_infinity}}
\frac{T_{2n}}{n} \stackrel{\text{п.н.}}\to \frac{1}{\lambda_+} + \frac{1}{\lambda_-}.
\end{align}

Более того, для любой константы $c_0 \in (0, \frac{1}{\lambda_+} + \frac{1}{\lambda_-})$, найдется \\$c_1 := c_1(c_0) > 0$: 
\begin{align}{\label{ET_2n_to_infinity}}
   \E e^{\lambda (c_0 n - T_{2n})} \le (1 - c_1\lambda)^n < 1, 
\end{align}
для достаточно малых $\lambda > 0$.
\end{lemma}
\begin{proof}[Доказательство леммы \ref{l1}.]
Заметим, что
\begin{align*}
T_{2n} = \underbrace{-0 + T_0}_{\Delta T_0} + \underbrace{- T_0 + T_1}_{\Delta T_1} + \underbrace{- T_1 + T_2}_{\Delta T_2} + \dots  + \underbrace{- T_{2n-1} + T_{2n}}_{\Delta T_{2n}} = T_0 + \sum _{i = 1}^{2n}\Delta T_i.
\end{align*}

Тогда $\{\Delta T_i + \Delta T_{i + 1}\}_{i=0}^{\infty}$ - независимые, одинаково распределенные случайные величины, причем для любого $n \in \N$  $n$-ый момент конечен,\\ $\E|\Delta T_i|^n < \infty,\, \forall i$, и также заметим, что $\E \Delta T_i + \Delta T_{i+1} = \frac{1}{\lambda+} + \frac{1}{\lambda-} > 0,\, \forall i \ge 0$.\\
Тогда по Усиленному закон больших чисел [\cite{link4}, глава 4, параграф 3], при $n \to \infty$, 
\begin{align*}
\frac{1}{n} T_{2n} \stackrel{\text{п.н.}}\to \frac{1}{\lambda+} + \frac{1}{\lambda-}.
\end{align*}
Что доказывает первое утверждение (\ref{T_n_to_infinity}) леммы \ref{l1}.

Положим $\Lambda = c_+ + c_-$, где $c_+ > 0$ и $c_- > 0$, их мы выберем позже. Используя неравенство Маркова и независимость скачков процесса $Z$, имеем для любого $\lambda > 0$,
\begin{eqnarray*}
\Prop\left(\frac{T_{2n}}{n} - \Lambda < -\epsilon_1 \right) &=& \Prop\left(-T_{2n} + \Lambda n > n\epsilon_1 \right)\\\\
&=& \Prop\left(e^{-\lambda T_{2n} + \Lambda \lambda n} > e^{\lambda\epsilon_1 n} \right)
\leq e^{-\lambda \epsilon_1 n} 	\E e^{-\lambda T_{2n} + \Lambda \lambda n}\\\\
&=&  e^{-\lambda \epsilon_1 n}  \E e^{-\lambda T_0} e^{-\lambda\left( \sum _{i = 1}^{n}(c_+ - \Delta T_{2i-1}) + \sum _{i = 1}^{n}(c_- - \Delta T_{2i})\right) }\\\\
&=& e^{-\lambda \epsilon_1 n}  \E e^{-\lambda T_0} \prod _{i = 1}^{n}\E e^{\lambda(c_+ - \Delta T_1)} \prod _{i = 1}^{n}\E e^{\lambda(c_- - \Delta T_2)}\\\\
&=& e^{-\lambda \epsilon_1 n} \E e^{-\lambda T_0} \left(\E  e^{\lambda(c_+ - \Delta T_1)} \E e^{\lambda(c_- - \Delta T_2)}\right) ^n.
\end{eqnarray*}
Заметим, что при $\lambda \to 0$, 
\begin{eqnarray*}
\E e^{\lambda (c_+ - \Delta T_{1})} 
&=&
\int  _{0}^\infty e^{\lambda(c_+ - x)}\lambda_+ e^{-\lambda_{+}x}\,dx = \frac{\lambda_{+} e^{\lambda c_+}}{\lambda + \lambda_{+}}\int _{0}^\infty(\lambda + \lambda_{+})e^{-(\lambda_{+} + \lambda)x}\, dx\\
&=& \frac{\lambda_{+}e^{\lambda c_+}}{(\lambda + \lambda_{+})} \stackrel{\text{Тейлор}}{=} 1 - (\frac{1}{\lambda_+} - c_+)\lambda + \overline{o}(\lambda),
\end{eqnarray*}
\begin{eqnarray*}
\E e^{\lambda (c_- - \Delta T_{2})} 
&=&
\int  _{0}^\infty e^{\lambda(c_- - x)}\lambda_+ e^{-\lambda_{-}x}\,dx = \frac{\lambda_{-} e^{\lambda c_-}}{\lambda + \lambda_{-}}\int _{0}^\infty(\lambda + \lambda_{-})e^{-(\lambda_{-} + \lambda)x}\, dx\\
&=& \frac{\lambda_- e^{\lambda c_-}}{\lambda + \lambda_-}=  1 - (\frac{1}{\lambda_-} - c_-)\lambda + \overline{o}(\lambda),\\
\E e^{-\lambda T_0} &=& \int  _{0}^\infty e^{-\lambda x}\lambda_- e^{-\lambda_{-}x}\, dx = \frac{\lambda_-}{\lambda + \lambda_-} \int\limits_0^{\infty}(\lambda + \lambda_-)e^{-(\lambda + \lambda_-)x}\, dx\\
&=&
\frac{\lambda_-}{\lambda + \lambda_-} = 1 - \frac{1}{\lambda_{-}}\lambda + \overline{o}(\lambda).
\end{eqnarray*}
В итоге получаем, для любого $0< \lambda < \min(\lambda_+, \lambda_-)$:
\begin{eqnarray}{\label{ET}}
\E e^{\lambda(\Lambda n - T_{2n})} &=& \frac{\lambda_-}{\lambda_- + \lambda} \left(\frac{\lambda_+ \lambda_-}{(\lambda_+ + \lambda)(\lambda_- + \lambda)}e^{\Lambda \lambda}\right)^n, \\
\label{ET-}
\E e^{\lambda(T_{2n} - \Lambda n )} &=& \frac{\lambda_-}{\lambda_- - \lambda} \left(\frac{\lambda_+ \lambda_-}{(\lambda_+ - \lambda)(\lambda_- - \lambda)}e^{-\Lambda \lambda}\right)^n.
\end{eqnarray}
Откуда, для любой константы $0< \Lambda < \frac{1}{\lambda_+} + \frac{1}{\lambda_-}$ найдется такая константа $c_1 > 0$, что при достаточно маленьком $\lambda$, 
\begin{align*}
\E e^{\lambda (\Lambda n -  T_{2n})} \le \left(1 - \frac{1}{\lambda_-} \lambda + \overline{o}(\lambda)\right) \left(1 - (\frac{1}{\lambda_+} + \frac{1}{\lambda_-} - \Lambda)\lambda + \overline{o}(\lambda) \right)^n \le (1 - c_1\lambda)^n < 1,\\
\end{align*}
Более того, для любого $\epsilon_1 > 0$, положив $\Lambda = \frac{1}{\lambda_+} + \frac{1}{\lambda_-}$, при $\, \lambda \to 0$ имеем следующие разложение,
\begin{eqnarray*}
\Prop\left(\frac{T_{2n}}{n} - \Lambda < -\epsilon_1 \right) \le \left( 1 - \frac{1}{\lambda_{-}}\lambda + \overline{o}(\lambda)\right) 
\left( 1 - \epsilon_1\lambda + \overline{o}(\lambda)\right)^n.\\
\end{eqnarray*}
В итоге, беря $\Lambda = \frac{1}{\lambda_{+}} + \frac{1}{\lambda_{-}}$, для любого положительного $\epsilon_1 \in R$ найдутся $\lambda_0 := \lambda_0(\epsilon_1, \Lambda)$ и $0< \kappa < 1$, такие, что для любого $\lambda > \lambda_0$
\begin{align}
\Prop\left(\frac{T_{2n}}{n} - \Lambda < -\epsilon_1 \right) \le \left(1 - \frac{1}{\lambda_-}\lambda \right)(1 - \epsilon_1\lambda)^n \le \kappa_1^n < 1.
\end{align}
Аналогично получается оценка. Для любого $\epsilon_2 > 0$ найдется $0 < \kappa_2 < 1$, такая, что
\begin{align}
\Prop\left(\frac{T_{2n}}{n} - \Lambda > \epsilon_2 \right) \le \kappa_2^n < 1.
\end{align}

\end{proof}
\begin{lemma} \label{l2}
В условиях теоремы найдется $c_1 > 0$, такая, что
\begin{align}{\label{l2_1}}
\liminf_{n \to \infty}{\frac{X_{T_{2n}}}{n}} \stackrel{\text{п.н.}} \ge c_1
\end{align}
Более того, верны экспоненциальные оценки. Найдется такая константа $c_1$, что для любого $ \epsilon > 0$, найдется $0 < \kappa_1 < 1$:
\begin{align} \label{l2_2}
\Prop_x\left(\frac{X_{T_{2n}}}{n} - c_1 < -\epsilon \right) \le \kappa_1^n < 1,
\end{align}
для достаточно больших $n > n_0(\epsilon, c_1)$.
\end{lemma}
\begin{proof}[Доказательство леммы \ref{l2}.]
Сначала получим оценки на вероятность ухода остановленного процесса $X_{T_{2n}}$ в момент времени $T_{2n}$ на бесконечность. После чего получим основной результат (\ref{l2_1}) леммы.
\par
Рассмотрим случай, где $Z_0 = 0$. Случай $Z_0 = 1$ рассматривается аналогично.\\
Пусть $a > 0$ произвольная постоянная, ее выберем позже. По формуле Ито:  
\begin{eqnarray*}
d e^{-\lambda X_{t} + at} &=& +ae^{-\lambda X_{t} + at} dt  -\lambda e^{-\lambda X_{t}+at} dX_t + \frac{\lambda^2}2 e^{-\lambda X_{t} + at} 
(dX_t)^2  \\
&=& +a e^{-\lambda X_{t} + at} dt  -\lambda e^{-\lambda X_{t} + at} (bdt + dW_t)\\
&&\; + \, \frac{\lambda^2}2 e^{-\lambda X_{t}+at} dt.
\end{eqnarray*}
В интегральной форме $\mathsf E := \E_x$, 
\begin{eqnarray*}\Large
\E e^{-\lambda (X_{T_2} - x) + a T_2} - 1 &=&
\E \int _0^{T_2} e^{-\lambda (X_{t} - x) + at} (a  -\lambda b_t  + \frac{\lambda^2}2) dt \\
&\le& 
\underbrace{(a-r_+\lambda) \E \int _0^{T_1} e^{-\lambda (X_t - x) + at}dt}_{=: I_1} + \\
&&\; +\, \underbrace{(a+r_-\lambda) \E \int _{T_1}^{T_2} e^{-\lambda (X_t - X_{T_1}) + at} dt}_{=: I_2}.
\end{eqnarray*}
Рассмотрим $t \in [0, T_1]$. Тогда в силу единственности решения $X_t$ совпадает с решением уравнения
\begin{align*}
\bar X_t - x = \int _{0}^{t} b_+(\bar X_s)\, ds + \int _{0}^{t} dW_s.
\end{align*}
При  $t \in [0, T_1]$ имеем:
\begin{align*}
\bar X_t = x + \int _{0}^{t} b_+(\bar X_s)\, ds + \int _{0}^{t} dW_s \stackrel{\text{п.н.}}{\geq}   x + \int _{0}^{t} r_+\, ds + \int _{0}^{t} dW_s =: \tilde{X_t}.
\end{align*}
С помощию этих неравенств получим оценки на величины:
\begin{align*}
\E\int _0^{T_1}e^{-\lambda (X_t - x) + at}\,dt,\;\;
\E\int _{T_1}^{T_2}e^{-\lambda (X_{T_2} - X_{T_1}) + at}\,dt.
\end{align*}
И, используя теорему Фубини и независимость $Z_t$ и $\cF^W$, получаем оценку на $I_{T_1} := \E\int _0^{T_1}e^{-\lambda (X_t - x) + at}dt$:
\begin{eqnarray*}\Large
I_{T_1} &\leq&  \E\int _0^{T_1}e^{-\lambda (\tilde{X_t} - x) + at}dt \\
&=& \E\int _0^{T_1}e^{(a - \lambda r_+)t  -\int _{0}^{t}\lambda dW_s}dt \\
&=& \E\int _0^{T_1}e^{(a - \lambda r_+)t) (-\int _{0}^{t}\lambda dW_s }dt \\
&=& \E\int _0^{\infty}1\{ t < T_1\}e^{(a + \frac{\lambda^2}{2} - \lambda r_+ )t}
e^{-\int _{0}^{t}\lambda dW_s - \frac{1}{2}\int _{0}^{t}\lambda^2 ds}dt \\
&=& \int _0^{\infty} \E 1\{ t < T_1\} e^{(a + \frac{\lambda^2}{2} - \lambda r_+ )t} \e^{-\int _{0}^{t}\lambda dW_s - \frac{1}{2}\int _{0}^{t}\lambda^2 ds}dt \\
&=&  \int _0^{\infty}  e^{(a + \frac{\lambda^2}{2} - \lambda r_+ )t} \E I\{t < T_1\} *1\, dt \\
&=& \int _0^{\infty}  e^{(a + \frac{\lambda^2}{2} - \lambda r_+ )t}e^{-\lambda_+ t}\, dt \\
&=&  \frac1{\lambda_+ -a+\lambda r_+ -\lambda^2/2}.
\end{eqnarray*}

Рассмотрим $t \in [T_1, T_2]$. Тогда в силу единственности решения, значения $X_t$ на этом отрезке совпадает со значением решения СДУ,
$$
\hat X_t = X_{T_1} + \int _{T_1}^{t} b_-(\hat X_s)\, ds + \int _{T_1}^{t} dW_s, \quad t\ge T_1.
$$
При  $t \in [T_1, T_2]$ имеем:
$$ 
\hat X_t = X_{T_1} + \int _{T_1}^{t} b_-(\hat  X_s)\, ds + \int _{T_1}^{t} dW_s \stackrel{\text{п.н.}}{\geq}  X_{T_1} + \int _{T_1}^{t} -r_-\, ds + \int _{T_1}^{t} dW_s =: \tilde{X_t}.
$$
Получим оценку на величину:
\begin{align*}
    \E_{\cF_{T_1}}(\int _{T_1}^{T_2}e^{-\lambda (X_t - x) + at}dt := \E\left( \int _{T_1}^{T_2}e^{-\lambda (X_t - x) + at}dt \vert X_{T_1} = x,\, T_1 = \tau\right).
\end{align*}
Использем замену времени в стохастическом итеграле:
\begin{eqnarray*}
\E_{{\cal F}_{T_1}}\int _{T_1}^{T_2}e^{-\lambda (X_t - x) + at}dt &\leq&  \E_{{\cal F}_{T_1}}\int _{T_1}^{T_2}e^{-\lambda (\tilde{X_t} - x) + at}dt 
\\	
&=& e^{a T_1}
\E_{{\cal F}_{T_1}}\int _{T_1}^{T_2}e^{-\lambda \tilde{X_t} + a(t-T_1) }dt
\\
&=& e^{-\lambda X_{T_1} + a T_1}
\E_{{\cal F}_{T_1}} \int_0^T 
e^{ at + \lambda r_- t- \lambda \tilde W_t}dt
\\
&=&
e^{-\lambda X_{T_1} + a T_1}
\frac{1}{\lambda_- -a- \lambda r_- -\lambda^2/2}.
\end{eqnarray*}

Тогда,
$$
\E\left( \E_{{\cal F}_{T_1}}\int _{T_1}^{T_2}e^{-\lambda (X_t - x) + at}dt\right)
\le
\frac{1}{\lambda_- -a- \lambda r_- -\lambda^2/2}\left(1 + 
\frac{a-r_+\lambda}{\lambda_+ -a+ \lambda r_+ -\lambda^2/2} \right).
$$
Поэтому, 
\begin{eqnarray*}
I_2 \le 
\frac{a+r_-\lambda}{\lambda_- -a- \lambda r_- -\lambda^2/2}
\left(1 + 
\frac{a-r_+\lambda}{\lambda_+ -a+ \lambda r_+ -\lambda^2/2} \right),
\end{eqnarray*}
и
\begin{eqnarray*}
I_1 \le 
\frac{a-r_+\lambda}{\lambda_+ -a+ \lambda r_+ -\lambda^2/2}\, .
\end{eqnarray*}
В итоге при достаточно малых $ a,\lambda > 0$, например, $a =\hat{a} \lambda,$ и при$ \lim _{n \to \infty} \lambda = 0,$ положив $0< \hat{a} < \frac{\lambda_{-}r_+ - \lambda_{+}r_-}{\lambda_{+} + \lambda_{-}}$, получаем:
\begin{eqnarray*}
\E_x e^{-\lambda X_{T_2} + aT_2} \le 1 + I_1 + I_2
\\
\le 1 + \underbrace{\frac{a-r_+\lambda}{\lambda_+ -a+ \lambda r_+ -\lambda^2/2}}_{(*)} +
\underbrace{\frac{a+r_-\lambda}{\lambda_- -a- \lambda r_- -\lambda^2/2}}_{(**)}
\left(1 + 
\underbrace{\frac{a-r_+\lambda}{\lambda_+ -a+ \lambda r_+ -\lambda^2/2}}_{(*)} \right)
\\
=
1 +\underbrace{\left[(\hat{a} - r_+) \lambda_{+}^{-1} + (\hat{a} + r_-)\lambda_{-}^{-1}) \right] }_{=-a_2} \lambda + \overline{o}(\lambda)\, .
\end{eqnarray*}
Оцени каждое из слагаемых по отдельности,
\begin{eqnarray*}
(*) &=& \frac{(a - r_+)\lambda_{+}^{-1}\lambda}{1 - \left[\frac{a - r_+}{\lambda_{+}} + \frac{1}{2\lambda_{+}}\lambda \right]\lambda}  = (a - r_+)\lambda_{+}^{-1}\lambda \left(1  + \left[\frac{a - r_+}{\lambda_{+}} + \frac{1}{2\lambda_{+}}\lambda \right]\lambda + \overline{o}(\lambda)\right)\, .\\
&=&(a - r_+)\lambda_{+}^{-1}\lambda + \overline{o}(\lambda)\\
(**) &=& \frac{(a + r_-)\lambda_{-}^{-1}\lambda}{1 - \left[\frac{a + r_-}{\lambda_{-}} + \frac{1}{2\lambda_{-}}\lambda \right]\lambda} = (a + r_-)\lambda_{-}^{-1}\lambda + \overline{o}(\lambda).
\end{eqnarray*}
Заметим, что в предположениях теоремы и при выполнении условия $0<\hat a<  \frac{\lambda_{-}r_+ - \lambda_{+}r_-}{\lambda_{+} + \lambda_{-}}$, получаем, что $a_2 := (\hat{a} - r_+) \lambda_{+}^{-1} + (\hat{a} + r_-)\lambda_{-}^{-1}$ больше нуля.
\par
В итоге имеем,
\begin{eqnarray*}
\E_x e^{-\lambda (X_{T_2}-x) + aT_2} \le 
1 - a_2\lambda + o(\lambda).
\end{eqnarray*}
Аналогично имеем оценку на $X_{T_4}$, 
\begin{eqnarray*}
\E_x e^{-\lambda (X_{T_4}-x) + aT_4} 
= \E e^{-\lambda (X_{T_2}-x) + aT_2} \times \underbrace{E\left(e^{-\lambda (X_{T_4} - X_{T_2}) + a(T_4 - T_2)}\vert {\cal F}_{T_2}\right)}_{\le 1 - a_2\lambda} \le (1 - a_2\lambda)^2. 
\end{eqnarray*}
Далее по индукции получаем оценку для произвольного $n \in \N$,
\begin{eqnarray}{\label{est_X_{T_{2n}}}}
\E_xe^{-\lambda (X_{T_{2n}}-x) + aT_{2n}} 
\le (1 - a_2\lambda)^n.
\end{eqnarray}

Используя неравенство Маркова, получим экспоненциальные оценки на $X_{T_{2n}}$, при $n > 1$,
\begin{align*}
\Prop_x\left(\frac{X_{T_{2n}}-x}{n} - c_1 < -\epsilon \right) 
= \Prop_x\left(-(X_{T_{2n}} - x) + c_1 n > n\epsilon \right)\\
\stackrel{\forall \lambda > 0 } 
= \Prop_x\left(e^{-\lambda(X_{T_{2n}} - x) + c_1 \lambda n)} > e^{\lambda\epsilon n} \right) 
\leq e^{-\lambda \epsilon n} \E\left(e^{-\lambda(X_{T_{2n}} - x) + c_1 \lambda n}\right) \, .
\end{align*}
Используя результаты леммы \ref{l1} и неравенство Коши-Буняковского-Шварца, имеем при достаточно больших $n$ следующую оценку на вероятность хода процесса на бесконечность.
\par
Положим $0 < c_1 < \min(\frac{1}{\lambda_+} + \frac{1}{\lambda_-},  \frac{\lambda_{-}r_+ - \lambda_{+}r_-}{\lambda_{+} + \lambda_{-}})$, беря произвольную $\epsilon > 0$, найдутся $\lambda :=\lambda(\epsilon, c_1) > 0$ и константы $c > 0$, и $0 < \kappa < 1$, такие, что
\begin{eqnarray*}
\Prop_x\left(\frac{X_{T_{2n}}-x}{n} - c_1 < -\epsilon \right)
&\leq& e^{-\lambda \epsilon n} \E\left(e^{-\lambda(X_{T_{2n}} - x) + c_1 \lambda n}\right) 
\nonumber 
\\
&\le& e^{-\lambda \epsilon n} \left( \E e^{-2\lambda(X_{T_{2n}} - x) + 2 c_1 \lambda T_{2n}} \E e^{2\lambda(c_1 n - T_{2n})}\right) ^{1/2}
\nonumber 
\\
&\le& (1 - \epsilon_1 \lambda +  \overline{o}(\lambda))^n \left( (1 - a_2 \lambda + \overline{o}(\lambda))^n (1 - c_1 \lambda + \overline{o}(\lambda))^n\right)^{\frac{1}{2}}
\nonumber 
\\
&\le& (1 - c\lambda)^n \le \kappa ^n < 1.
\end{eqnarray*}
Что и завершает доказательства второго утверждения (\ref{l2_2}) леммы \ref{l2}.
\par
Докажем первое утверждение (\ref{l2_1}) леммы \ref{l2}.
\par
Существует такая $c_1 > 0$, что при достаточно малом положительном $\lambda\,, \lambda := \lambda(\epsilon, c_1)$, имеем:
\begin{align*}
\sum _{n = 1}^{\infty}\Prop_x\left(-(X_{T_{2n}}-x) + n c_1 > n\epsilon \right)  < \infty.
\end{align*}

По лемме Бореля-Кантелли из полученного неравенства следует, что
$$
\Prop_x(-(X_{T_{2n}} - x) + n c_1 \leq \epsilon n \, \text{бесконечно часто}) = 0.
$$
Поэтому, для любого $\epsilon > 0$ для п.в. $\omega \in \Omega$, существует $N_0 = N_0(\omega)$:
\begin{align*}
X_{T_{2n}}-x \geq nc_1 - \epsilon n  \;\;\; \forall n \geq N_0, 
\end{align*}
или
\begin{align*}
\liminf_{n \to \infty}{\frac{X_{T_{2n}}}{n}} \stackrel{\text{п.н.}} \ge c_1.
\end{align*}
Аналогично рассматривается случай, когда $Z_0 = 1$.\\
Лемма доказана.
% Итог:  $\liminf_{n \to \infty}{\frac{X_{T_{2n}}}{n}} \stackrel{\text{п.н.}} \ge c_1 > 0$.
\end{proof}

\begin{proof}[Доказательство теоремы \ref{t1}.]
Рассмотрим $k(t) = \lceil \frac{t}{\Lambda} \rceil$. $\Lambda > 0 $ будет выбрана нами позже. Заметим справедливость неравенства: $ \frac{t}{\Lambda} \le \lceil \frac{t}{\Lambda} \rceil < \frac{t}{\Lambda} + 1$.
Так как $k(t)$ неслучайная, $T_{2k(t)}$ - марковский момент остановки.\\
Рассмотрим следующее представление процесса $X_t$,
\begin{align*}
X_t = X_{T_{2k(t)}} + (X_t - X_{T_{2k(t)}}).
\end{align*}
Из леммы \ref{l2} и вида процесса $k(t)$ следует, что
\begin{align*}
\liminf\limits_{t \to \infty}\frac{X_{T_{2k(t)}}}{t} = \liminf\limits_{t \to \infty}\frac{X_{T_{2k(t)}}}{k(t)}\frac{k(t)}{t} \stackrel{\text{п.н.}}\ge c_1  \frac{1}{\Lambda} > 0
\end{align*}
Поэтому остается рассмотреть слагаемое: $X_t - X_{T_{2k(t)}}$.\\
\begin{eqnarray*}
\frac{X_t - X_{T_{2k(t)}}}{t} &=& \frac{\int\limits_{0}^{t}b(X_s, Z_s)\, ds + \int\limits_{0}^t\, dW_s - \left(\int\limits_{0}^{T_{2k(t)}}b(X_s, Z_s)\, ds + \int\limits_{0}^{T_{2k(t)}}\, dW_s\right)}{t} \\
&\stackrel{\text{п.н.}}\ge& \frac{-(r_- + r_+) |t - T_{2k(t)}| + W_t - W_{T_{2k(t)}}}{t}.
\end{eqnarray*}
Заметим, используя лемму \ref{l1} и закон повторного логарифма \cite{link9}:
\begin{eqnarray*}
\liminf\limits_{t \to \infty}-(r_+ + r_-)\frac{|t - T_{2k(t)}|}{t} &=& -(r_+ + r_-)\liminf\limits_{t \to \infty} \frac{k(t)}{t}|\frac{T_{2k(t)}}{k(t)} - \frac{t}{k(t)}| \stackrel{\text{п.н.}}= 0,\\
\liminf_{t \to \infty}{\frac{W_t - W_{T_{2k(t)}}}{t}} \stackrel{\text{п.н.}}= 0.
\end{eqnarray*}
Откуда и получаем первый результат теоремы (\ref{t1_1}).

Установим экспоненциальную оценку \ref{t1_2}.\\
Используя неравенство Маркова и неравество Гельдера, для любого $\epsilon_1 > 0$ и некоторого $c_0 > 0$ и достаточно малое положительное $\lambda$ получим следующую оценку:
\begin{eqnarray*}
\Prop_x\left(\frac{X_t - x}{t} - c_0 < -\epsilon_1\right)
&\le&
e^{-\lambda \epsilon_1 t}\E e^{-\lambda (X_t - x) + \lambda c_0 t} = e^{-\lambda \epsilon_1 t}\E e^{-\lambda(X_{T_{2k(t)}} - x + (X_t - X_{T_{2k(t)}})) + \lambda c_0 t)}\\
&\le&
e^{-\lambda \epsilon_1 t}\left[\E e^{-2\lambda(X_t - X_{T_{2k(t)}})}\right]^{\frac{1}{2}}\left[\E e^{-2\lambda (X_{T_{2k(t)}} - x) + 2\lambda c_0 t} \right]^{\frac{1}{2}}.
\end{eqnarray*}
Используя неравество Гельдера и свойства стохастической экспоненты получаем следующую оценку:
\begin{eqnarray*}
\E e^{-2\lambda(X_t - X_{T_{2k(t)}})} &\le& \E e^{2\lambda\left((r_- + r_+)(t \vee T_{2k(t)} - t \wedge T_{2k(t)}) - (W_{t} - W_{T_{2k(t)}}\right)}\\ 
&\le&
\left[\E e^{4\lambda (r_- + r_+) |t - T_{2k(t)}|}\right]^{\frac{1}{2}} \left[\E e^{-4\lambda(W_{t} - W_{T_{2k(t)}})}\right]^\frac{1}{2}\\
&=&
\left[\E e^{4\lambda (r_- + r_+) |t - T_{2k(t)}|}\right]^{\frac{1}{2}} \left[\E e^{-8\lambda W_{t - T_{2k(t)} \wedge t}}\right]^\frac{1}{4}\left[\E e^{8\lambda W_{T_{2k(t)} - T_{2k(t)} \wedge t}}\right]^\frac{1}{4}\\
&\le&
\left[\E e^{4\lambda (r_- + r_+) |t - T_{2k(t)}|}\right]^{\frac{1}{2}} \left[\E e^{\int\limits_{0}^{t - T_{2k(t)}\wedge t}-8\lambda\, dW_s \mp \frac{1}{2}\int\limits_{0}^{t - T_{2k(t)}\wedge t} 64\lambda^2\, ds}\right]^\frac{1}{4} * \\
&&\; *\, \left[\E e^{8\lambda W_{T_{2k(t)} - T_{2k(t)} \wedge t}}\right]^\frac{1}{4}\\
&\le&
\left[\E e^{4\lambda (r_- + r_+) |t - T_{2k(t)}|}\right]^{\frac{1}{2}} \left[1 * \left[\E e^{64\lambda^2 |t - T_{2k(t)}|} \right]^{\frac{1}{2}} \right]^{\frac{1}{2}}.
\end{eqnarray*}
Используя промежуточный результат леммы \ref{l2} (\ref{est_X_{T_{2n}}}) и неравенство Гельдера, получим следующую оценку для достаточно маленьких положительных $\lambda$ при $0 < 4 c_0 < \min(\frac{1}{\lambda_+} + \frac{1}{\lambda_-},  \frac{\lambda_{-}r_+ - \lambda_{+}r_-}{\lambda_{+} + \lambda_{-}})$:
\begin{eqnarray*}
\E e^{-2\lambda (X_{T_{2k(t)}} - x) + 2\lambda c_0 t} &=& \E e^{-2\lambda (X_{T_{2k(t)}} - x) + c_0 \lambda T_{2k(t)} - 2\lambda c_0(t - T_{2k(t)})}\\
&\le&
\left[e^{-4 \lambda (X_{T_{2k(t)}} - x) + 4 c_0 \lambda T_{2k(t)}}\right]^{\frac{1}{2}} \left[\E e^{-4 \lambda c_0 (t - T_{2k(t)})}\right]^{\frac{1}{2}}\\
&\le&
\left[(1 - 4c_0\lambda)^{k(t)}\right]^{\frac{1}{2}} \left[\E e^{-4 \lambda c_0 (t - T_{2k(t)})}\right]^{\frac{1}{2}}.
\end{eqnarray*}
В итоге имеем следующую оценку,
\begin{eqnarray*}
\Prop_x\left(\frac{X_t - x}{t} - c_0 < -\epsilon_1\right)&\le& e^{-\lambda \epsilon_1 t}\left[\E e^{-2\lambda(X_t - X_{T_{2k(t)}})}\right]^{\frac{1}{2}}\left[\E e^{-2\lambda (X_{T_{2k(t)}} - x) + 2\lambda c_0 t} \right]^{\frac{1}{2}}\\
&\le&
e^{-\lambda \epsilon_1 t} \left[\E e^{4\lambda (r_- + r_+) |t - T_{2k(t)}|}\right]^{\frac{1}{4}} \left[\E e^{16\lambda^2 |t - T_{2k(t)}|} \right]^{\frac{1}{8}}*\\
&& \; *\, \left[(1 - 4c_0\lambda)^{\frac{k(t)}{2}}\left[\E e^{-4 \lambda c_0 (t - T_{2k(t)})}\right]^{\frac{1}{2}}\right]^{\frac{1}{2}}\\
&=&
e^{-\lambda \epsilon_1 t} \left[\E e^{4\lambda (r_- + r_+) |t - T_{2k(t)}|}\right]^{\frac{1}{4}} \left[\E e^{16\lambda^2 |t - T_{2k(t)}|} \right]^{\frac{1}{8}} \\
&& \; * \, \left[\E e^{-4 \lambda c_0 (t - T_{2k(t)})}\right]^{\frac{1}{4}}\left[(1 - 4c_0\lambda)\right]^{\frac{k(t)}{4}}
\end{eqnarray*}
 Используя оценки (\ref{ET}, \ref{ET-}) из леммы \ref{l1} и неравенство $e^{|a - b|} \le e^{a - b} + e^{b - a}$, положив $\Lambda = \frac{1}{\lambda_+} + \frac{1}{\lambda_-}$, найдем такую константу $K > 0$, что при $\lambda \to 0+$ будет выполнено неравенство:
\begin{eqnarray*}
\E e^{\lambda |t - T_{2k(t)}|} &\le& e^{\lambda(t - \Lambda k(t))}\E e^{\lambda (\Lambda k(t) - T_{2k(t)})} + e^{\lambda(\Lambda k(t) - t)} \E e^{\lambda (T_{2k(t)} - \Lambda k(t))}\\
&=& 
e^{\lambda(t - \Lambda k(t))}\frac{\lambda_-}{\lambda_- + \lambda} \left(\frac{\lambda_+ \lambda_-}{(\lambda_+ + \lambda)(\lambda_- + \lambda)}e^{\Lambda \lambda}\right)^{k(t)} +\\
&& \; + \, e^{\lambda(\Lambda k(t) - t)} \frac{\lambda_-}{\lambda_- - \lambda} \left(\frac{\lambda_+ \lambda_-}{(\lambda_+ - \lambda)(\lambda_- - \lambda)}e^{-\Lambda \lambda}\right)^{k(t)}\\
&=&
e^{\lambda t} \frac{\lambda_-}{\lambda_- + \lambda} \left(\frac{\lambda_+ \lambda_-}{(\lambda_+ + \lambda)(\lambda_- + \lambda)}\right)^{k(t)} +\\
&& \; + \, e^{-\lambda t} \frac{\lambda_-}{\lambda_- - \lambda} \left(\frac{\lambda_+ \lambda_-}{(\lambda_+ - \lambda)(\lambda_- - \lambda)}\right)^{k(t)} \\
&=&
e^{\lambda t}  \frac{\lambda_-}{\lambda_- + \lambda} \left[1 - \Lambda \lambda + \overline{o}(\lambda)\right]^{k(t)} + e^{-\lambda t}  \frac{\lambda_-}{\lambda_- - \lambda} \left[1 + \Lambda \lambda + \overline{o}(\lambda)\right]^{k(t)}\\
&\le&
e^{\lambda t}  \frac{\lambda_-}{\lambda_- + \lambda} \left[1 - \Lambda \lambda + \overline{o}(\lambda)\right]^{\frac{t}{\Lambda} + 1} + e^{-\lambda t}  \frac{\lambda_-}{\lambda_- - \lambda} \left[1 + \Lambda \lambda + \overline{o}(\lambda)\right]^{\frac{t}{\Lambda}}\\
&\le&
K e^{\lambda t}\left[1 - \lambda + \overline{o}(\lambda)\right]^t + K \left[1 + \lambda  \overline{o}(\lambda)\right]^t\\
&=&
K \left[1 + \overline{o}(\lambda)\right].
\end{eqnarray*}

В итоге для любого $\epsilon_1 > 0$ и $\Lambda = \frac{1}{\lambda_+} + \frac{1}{\lambda_-}$ при $\lambda \to 0+$ найдется положительные константы $\hat{K}, c_0$ для достаточно больших $t$:
\begin{eqnarray*}
\Prop_x\left(\frac{X_t - x}{t} - c_0 < -\epsilon_1\right) \le \hat{K} \e^{-\lambda \epsilon_1 t} \left[1 + \overline{o}(\lambda)\right]^{t}\left[1 - 4 c_0 \lambda \right]^{\frac{k(t)}{4}}.
\end{eqnarray*}
Или существует константа $c_0 >0$ такая, что для любого $\epsilon_1 > 0$ найдется $ 0< \kappa < 1$ и при достаточно больших $t > t_0(\lambda, \epsilon_1)$ выполнена оценка:
\begin{align*}
\Prop_x\left(\frac{X_t - x}{t} - c_0 < -\epsilon_1\right) < \kappa^t < 1.
\end{align*}
Теорема доказана.

\end{proof}

% \addcontentsline{toc}{section}{Заключение}
% \section*{Заключение}
% % \section{Заключение}
% Данная работа показывает ограничения на условия для экспоненциальной эргодичности с данной системой переключения.
% \par
% В работе установлены достаточные условия транзиентности процесса для модели марковской диффузии с переключениями и двумя режимами, транзиентным и эргодическим, при интенсивностях строго отделимых от нуля. Также получены экспоненциальные оценки вероятности ухода процесса на бесконечность при выполнении условий теоремы \ref{t1}.  

\section*{Благодарности}
Автор благодарен профессору Веретенникову А. Ю. 
%и анонимному рецензенту 
%(anonymous - лучше, чем undisclosed)
за постановку задачи и внимание к работе.

\end{document}